\documentclass{article}
\usepackage{amsmath,amssymb,amsthm}
\usepackage{tikz}
\usepackage{color}
\usepackage[toc]{appendix}
\usepackage{graphicx}
\usepackage{fancyhdr}
\usepackage{enumitem}
\usepackage{bbm}
\usepackage{parskip}
\usepackage{float}
\usepackage{chngpage}
\usepackage{calc}
\usepackage{bigints}
\usepackage{array}
\usepackage{booktabs}
\usepackage{rotating}
\usepackage{multirow}
\usepackage{adjustbox}
\usepackage{tabularx}
\usepackage{verbatim}
\usepackage{mathtools}
\usepackage{ragged2e}
\usepackage[makeroom]{cancel}
\usepackage{caption}
\usepackage{hyperref}
\usepackage{caption}
\usepackage{subcaption}
\usepackage{appendix}
\usepackage{pgfplots}
\pgfplotsset{compat=1.16}
\usepackage[english]{babel}
\usepackage{hyphenat}
\usepackage[makeindex]{imakeidx}
\usepackage{amsmath}

\newcommand{\Ai}{\operatorname{Ai}}
\newcommand{\Bi}{\operatorname{Bi}}

\newcommand{\vertiii}[1]{{\left\vert\kern-0.25ex\left\vert\kern-0.25ex\left\vert #1 
		\right\vert\kern-0.25ex\right\vert\kern-0.25ex\right\vert}}

\usetikzlibrary{datavisualization}
\usetikzlibrary{matrix}
\usetikzlibrary{datavisualization.formats.functions}

\newtheorem{theorem}{Theorem}[section]

\newtheorem{corollary}[theorem]{Corollary}
\newtheorem{lemma}[theorem]{Lemma}

\newtheorem{remark}[theorem]{Remark}
\newtheorem*{notation*}{Notation}

\newcount\madechanges
\def\caution#1{\ifnum \madechanges=1 \affixmessage{#1}%
	\else \relax \fi}
\def\affixmessage#1{\marginpar{{\footnotesize  \em #1} \openup
		-.3\baselineskip }}
\makeatletter
\def\section{\@startsection {section}{1}{\z@}{3.25ex plus 1ex minus
		.2ex}{1.5ex plus .2ex}{\large\bf}}
\def\subsection{\@startsection{subsection}{2}{\z@}{3.25ex plus 1ex minus
		.2ex}{1.5ex plus .2ex}{\normalsize\bf}}
\@addtoreset{equation}{section} 
\makeatother
\title{The Tricomi equation in the hyperbolic half plane under additive space-time Gaussian White Noise perturbation}

\author{Enrico Bernardi \thanks{Dipartimento di Scienze Statistiche Paolo Fortunati, Università di Bologna, Bologna, Italy.\\ \textbf{e-mail}: enrico.bernardi@unibo.it} \and Alberto Lanconelli\thanks{Dipartimento di Scienze Statistiche Paolo Fortunati, Università di Bologna, Bologna, Italy.\\ \textbf{e-mail}: alberto.lanconelli2@unibo.it}}
\date{\today}
\begin{document}
	\maketitle
	\bigskip
	\begin{abstract}
We study the Cauchy problem for the Tricomi equation perturbed by space-time Gaussian White Noise. To prove existence and uniqueness of the solution, we employ a Fourier transform approach that allows to obtain its representation in terms of certain integrals of the Airy functions. Then, via a careful analysis of the asymptotic behaviour of those integrals, we obtain all the desired properties of the solution, such as square integrability, continuity of its sample paths and stationarity with respect to the space variable. In relation to that stationarity, we also provide the precise description of how the correlation function behaves for small values of the space-lag. We also remark that, in contrast to the findings of the recent paper \cite{EnricoAlberto2}, the properties of the solution to our stochastic Tricomi equation are equivalent to those derived in studying the corresponding problem for the wave operator.  
	\end{abstract}
	
	

\allowdisplaybreaks

\section{Introduction and statement of the main results} 

The aim of this paper is to study the following Cauchy problem: 
\begin{align}\label{1bis}
\begin{cases}
\partial_{tt} U(t,x)=t\partial_{xx}U(t,x)+W(t,x),& t>0,x\in\mathbb{R};\\
U(0,x)=0,\quad \partial_{t}U(0,x)=0,& x\in\mathbb{R},
\end{cases}
\end{align}
where $\{W(t,x)\}_{t\geq 0,x\in\mathbb{R}}$ is a space-time Gaussian White Noise. The first equation in \eqref{1bis} represents a random perturbation of the classic Tricomi equation and we aim at investigating existence, uniqueness as well as probabilistic and analytic properties of its solution. \\
The Tricomi equation \cite{tricomi} is a second-order partial
differential equation of mixed elliptic-hyperbolic type with the form:
\begin{align}\label{tricomi intro}
\partial_{tt}u(t,x)=t\partial_{xx}u(t,x),\quad t,x\in\mathbb{R}.
\end{align}
The equation is hyperbolic in the half plane $t > 0$, elliptic in the half plane $t < 0$, and degenerates on the line $t = 0$. Many important problems in fluid mechanics and differential geometry can be reduced to corresponding problems for the Tricomi equation, particularly \emph{transonic flow} problems \cite{manwell} and \emph{isometric embedding} problems \cite{Qing} (see also \cite{tricomiintro} and the references quoted there). 

Most of the literature on hyperbolic partial differential equations perturbed by noise is centred on the wave equation. In the seminal works \cite{Walsh} and \cite{Orsingher} the authors consider the Cauchy problem for the one dimensional wave equation perturbed by additive space-time Gaussian White Noise proving existence, uniqueness and probabilistic properties of the solution. For higher space dimension, \cite{Dalang99} identifies a necessary and sufficient condition  relating additive noise terms to space dimension that allow for well posedness of the corresponding stochastic partial differential equations; see also \cite{DalangFrangos}. The so-called Hyperbolic Anderson model, i.e. wave equation perturbed by a multiplicative noise term with linear diffusion coefficient, is considered in several recent papers: see for instance \cite{BalanSong}, \cite{BalanJolis} and the reference quoted there. For non linear problems some relevant pointers are \cite{ConusDalang},\cite{Peszat} and \cite{PeszatZabczyk}. 

Stochastic partial differential equations with non strictly hyperbolic operators are not very much explored in the literature.  Some recent papers have attempted to investigate how the degeneracies of the operator have an impact on stochastic perturbations. In \cite{Ascanelli2} the authors study mild solutions of a class of stochastic partial differential equations, involving operators with polynomially bounded coefficients under suitable hyperbolicity hypotheses; \cite{Ascanelli3} and \cite{Ascanelli1} investigate the existence of random-field solutions for weakly hyperbolic stochastic partial differential equations; in \cite{EnricoAlberto} and \cite{EnricoLeonardo} the authors study the effect of Gaussian perturbations on a hyperbolic partial differential equation with double characteristics and double symplectic characteristics, respectively, in
low spatial dimensions. Furthermore, \cite{EnricoAlberto2} considers a class of Tricomi-type partial differential equations and analyze the robustness of the solution by taking the initial data to be Gaussian White Noise: the authors  discover that the existence of a well-defined random field solution is lost upon the introduction of lower-order terms in the operator.

In the present paper we focus on the Cauchy problem \eqref{1bis} and investigate existence and uniqueness of square integrable solutions. We recall that in the series of papers \cite{Gelfand1,Gelfand2,Gelfand2bis,Gelfand3} the authors have found an explicit expression for the fundamental solution of the Tricomi operator. However, due to the complexity of such representation, our approach avoids the use of that result and instead employs a Fourier method. Specifically, we consider the Cauchy problem	
\begin{align}\label{tricomi det}
\begin{cases}
\partial_{tt} u(t,x)=t\partial_{xx}u(t,x)+h(t,x),& t>0,x\in\mathbb{R};\\
u(0,x)=0,\quad \partial_{t}u(0,x)=0,& x\in\mathbb{R},
\end{cases}
\end{align}
for some smooth function $h:[0,+\infty[\times\mathbb{R}\to\mathbb{R}$. Taking the Fourier transform in the spatial variable and setting
\begin{align*}
\widehat{u}(t,\xi) = \int_{\mathbb{R}} e^{-i\xi x} u(t,x)dx\quad\mbox{ and }\quad\widehat{h}(t,\xi) = \int_{\mathbb{R}} e^{-i\xi x} h(t,x)dx,
\end{align*}
problem \eqref{tricomi det} transforms into 
\begin{align}\label{tricomi det 2}
\begin{cases}
\partial_{tt} \widehat{u}(t,\xi)=-t\xi^2\widehat{u}(t,\xi)+\widehat{h}(t,\xi),& t>0,\xi\in\mathbb{R};\\
\widehat{u}(0,\xi)=0,\quad \partial_{t}\widehat{u}(0,\xi)=0,& \xi\in\mathbb{R}.
\end{cases}
\end{align}
Now, the homogeneous equation
\begin{align*}
\partial_{tt} v(t,\xi)=-t\xi^2v(t,\xi)
\end{align*}
can be reduced by the change of variable $\eta := |\xi|^{2/3} t$ to
\begin{align*}
\partial_{\eta\eta} v(\eta,\xi)=-\eta v(\eta,\xi).
\end{align*}
The linearly independent solutions of this equation are the Airy functions, denoted $\mathrm{Ai}$ and $\mathrm{Bi}$, whose Wronskian is given by
\begin{align*}
\mathrm{Ai}(0)\mathrm{Bi}'(0) - \mathrm{Ai}'(0)\mathrm{Bi}(0) = \frac{1}{\pi}.
\end{align*}
Therefore, if we set
\begin{align*}
v_1(t,\xi): = \mathrm{Ai}\!\left(-|\xi|^{2/3} t\right)\quad\mbox{ and }\quad
v_2(t,\xi): = \mathrm{Bi}\!\left(-|\xi|^{2/3} t\right),
\end{align*}
whose Wronskian with respect to $t$ is
\begin{align*}
v_1(0,\xi) \partial_tv_{2}(0,\xi) - \partial_tv_1(0,\xi)v_{2}(0,\xi) = -\frac{|\xi|^{2/3}}{\pi},
\end{align*}
we obtain by variation of parameters the solution to \eqref{tricomi det 2} as
\begin{align}\label{eq:u-hat}
\widehat{u}(t,\xi)=  \int_{0}^{t} C(t,s,|\xi|)\,\widehat{h}(s,\xi)\,ds,
\end{align}
where, for \(\rho>0\),
\begin{align}\label{eq:C}
C(t,s,\rho):= -\frac{\pi}{\rho^{2/3}}\!
\left[ \mathrm{Bi}(-\rho^{2/3}t)\,\mathrm{Ai}(-\rho^{2/3}s)
- \mathrm{Ai}(-\rho^{2/3}t)\,\mathrm{Bi}(-\rho^{2/3}s)\right], 
\end{align}
and we set
\begin{align*}
 C(t,s,0):=t-s,
\end{align*}
which is the continuous extension of \eqref{eq:C} to \(\rho=0\).
Thus, inverting the Fourier transform we can represent the solution to \eqref{tricomi det} as
\begin{align}\label{eq:solution}
u(t,x) = \frac{1}{2\pi}\int_{\mathbb{R}} e^{i\xi x}
\int_{0}^{t} C(t,s,|\xi|)\widehat{h}(s,\xi)dsd\xi.
\end{align}
Formula \eqref{eq:solution} will be the starting point for our study of \eqref{1bis}.

To formulate our main result, we now assume that space-time Gaussian White Noise $\{W(t,x)\}_{t\geq 0,x\in\mathbb{R}}$ from \eqref{1bis} is represented through the formal series 
\begin{align}\label{serie}
W(t,x)=\sum_{m,n\geq 1}\psi_m(t)\varphi_n(x)\mathtt{z}_{mn},\quad t\geq 0,x\in\mathbb{R}.
\end{align} 
Here:
\begin{itemize}
\item $\{\psi_m\}_{m\geq 1}$ is a fixed real, smooth orthonormal basis
  of $L^2([0,+\infty))$, for instance, one may take a Laguerre basis;
\item $\{\varphi_n\}_{n\geq 1}$ is the orthonormal basis of $L^2(\mathbb{R})$ made of Hermite functions;
\item $\{\mathtt{z}_{mn}\}_{m,n\geq 1}$ denotes a family of independent standard Gaussian random variables defined on a common probability space $(\Omega,\mathcal{F},\mathbb{P})$.
\end{itemize}
A solution $\{U(t,x)\}_{t\geq 0,x\in\mathbb{R}}$ to problem \eqref{1bis} is by definition 
\begin{align*}
U(t,x):=\lim_{N\to+\infty}U_N(t,x)\quad\mbox{ in }\mathbb{L}^2(\Omega,\mathcal{F},\mathbb{P})
\end{align*}
for all $t\geq 0$ and $x\in\mathbb{R}$, where $\{U_N(t,x)\}_{t\geq 0,x\in\mathbb{R}}$ is solution to
\begin{align}\label{1bis_N}
\begin{cases}
\partial_{tt}U_N(t,x)=t\partial_{xx}U_N(t,x)+W_N(t,x),& t>0,x\in\mathbb{R};\\
U_N(0,x)=0,\quad \partial_{t}U_N(0,x)=0,& x\in\mathbb{R},
\end{cases}
\end{align}  
while
\begin{align*}
W_N(t,x)=\sum_{m,n= 1}^N\psi_m(t)\varphi_n(x)\mathtt{z}_{mn},\quad t\geq 0,x\in\mathbb{R}, 
\end{align*} 
provides a finite-rank smooth approximation of the space-time Gaussian White Noise $\{W(t,x)\}_{t\geq 0,x\in\mathbb{R}}$. Notice that problem \eqref{1bis_N} can be treated $\omega$-wise as a non-homogeneous deterministic Tricomi equation of the form \eqref{tricomi det}. \\
We are now ready to state the main result of our paper.

\begin{theorem}\label{main theorem}
Problem \eqref{1bis} admits a unique solution, in the above mild random-field sense, which is a Gaussian random field with representation
\begin{align}\label{solution series}
U(t,x)=\sum_{m,n\geq 1}u_{mn}(t,x)\mathtt{z}_{mn},\quad t\geq 0,x\in\mathbb{R},
\end{align}
where
\begin{align*}
u_{mn}(t,x):=\frac{(-i)^n}{2\pi}\int_{\mathbb{R}}\int_{0}^te^{ix\xi}C(t,s,|\xi|)\psi_m(s)\varphi_n(\xi)dsd\xi.
\end{align*}
Here, for \(\rho>0\),
\begin{align}\label{def C}
C(t,s,\rho):= -\frac{\pi}{\rho^{2/3}}\!
\left[ \mathrm{Bi}(-\rho^{2/3}t)\,\mathrm{Ai}(-\rho^{2/3}s)
- \mathrm{Ai}(-\rho^{2/3}t)\,\mathrm{Bi}(-\rho^{2/3}s)\right],\quad s,t\geq 0, 
\end{align}
and \(C(t,s,0):=t-s\). The functions $\mathrm{Ai}$ and $\mathrm{Bi}$ are the Airy functions. The series in \eqref{solution series} converges in $\mathbb{L}^2(\Omega,\mathcal{F},\mathbb{P})$ and almost surely. Furthermore,
\begin{align}\label{norma}
\mathbb{E}[|U(t,x)|^2]&=\frac{1}{2\pi^2}\int_0^{+\infty}\int_0^{t}\left|C(t,s,\rho)\right|^2ds d\rho<+\infty,\quad t\geq 0,x\in\mathbb{R}.
\end{align}
\end{theorem}

The proof of Theorem \ref{main theorem} is postponed to Section \ref{proof main theorem} below. The next result describes probabilistic and analytic properties of the stochastic process $x\mapsto U(t,x)$ for fixed $t>0$; in particular, we show that it is second order stationary with correlation function admitting a specific asymptotic behaviour that guarantees continuity of its paths. Such property is shared with the stochastic wave equation, as proved in \cite{Orsingher}.

\begin{corollary}\label{corollary}
For any $t>0$ the stochastic process $\{U(t,x)\}_{x\in\mathbb{R}}$ is second order stationary with expectation
\begin{align*}
\mathbb{E}[U(t,x)]=0\quad\mbox{ for all }x\in\mathbb{R}
\end{align*}
and covariance
\begin{align}\label{cov}
\mathtt{cov}(U(t,x),U(t,y))=\frac{1}{2\pi^2}\int_0^{+\infty}\cos(\rho |x-y|)\int_0^t|C(t,s,\rho)|^2dsd\rho\quad\mbox{ for all }x,y\in\mathbb{R}.
\end{align}
Moreover, the correlation function
\begin{align*}
r_t(h):=\frac{\mathtt{cov}(U(t,x+h),U(t,x))}{\mathbb{V}[U(t,x)]},\quad t>0,x,h\in\mathbb{R},
\end{align*}
verifies
\begin{align*}
r_t(h)=1-\alpha_t|h|+o(|h|)\mbox{ as $|h|$ tends to zero},
\end{align*}
with
\begin{align*}
\alpha_t:=\frac{\pi}{2\int_0^{+\infty}\int_0^{t}\left|C(t,s,\rho)\right|^2ds d\rho}.
\end{align*}
In particular, the stochastic process $\{U(t,x)\}_{x\in\mathbb{R}}$ admits a continuous modification.
\end{corollary}

The proof of Corollary \ref{corollary} can be found in Section \ref{proof corollary} below. Our last result concerns the continuity property of the stochastic process $t\mapsto U(t,x)$ for fixed $x\in\mathbb{R}$. Here, in absence of stationarity, we must resort to the Kolmogorov continuity theorem which, in combination with the Gaussian nature of our solution process, amounts at estimating the second moment only. In fact, for any $p\geq 1$
 there exists a positive constant $C_p$ such that
 \begin{align*}
 \mathbb{E}[|U(t,x)-U(s,x)|^p]=C_p \mathbb{E}[|U(t,x)-U(s,x)|^2]^{\frac{p}{2}},\quad s,t>0,x\in\mathbb{R},
 \end{align*}
and hence an estimate of the type
\begin{align*}
\mathbb{E}[|U(t,x)-U(s,x)|^2]\leq C|t-s|
\end{align*}
will be sufficient for the purpose. For the proof of Corollary \ref{corollary2} see Section \ref{proof corollary2} below.

\begin{corollary}\label{corollary2}
	For any $x\in\mathbb{R}$ the stochastic process $\{U(t,x)\}_{t\geq 0}$ admits a continuous modification with H$\ddot{o}$lder regularity of order less than $1/2$.
\end{corollary}

The paper is organized as follows: Section \ref{proof main theorem} presents the proof of Theorem \ref{main theorem} where we exploit the explicit representation of the White Noise given in \eqref{serie} and some basic facts on Airy functions; in Section \ref{proof corollary} we prove Corollary \ref{corollary}: the proof is articulated in several fine estimates involving the Airy functions, for negative values of their arguments, and related integrals: this will be the key for obtaining the desired properties of the solution; lastly, Section \ref{proof corollary2} contains the proof of Corollary \ref{corollary2} which amounts, by virtue of Kolmogorov continuity theorem and the Gaussianity of the solution, at proving an upper bound for the second moment of the time increment of the solution. 

\section{Proof of Theorem \ref{main theorem}}\label{proof main theorem}

Referring to the representation formula \eqref{eq:solution}, we take
\begin{align}\label{data}
h(t,x):=W_N(t,x)=\sum_{m,n= 1}^N\psi_m(t)\varphi_n(x)\mathtt{z}_{mn},\quad t\geq 0,x\in\mathbb{R}, 
\end{align}
which gives
\begin{align*}
\widehat{h}(s,\xi)&=\sum_{m,n= 1}^N\psi_m(s)\widehat{\varphi}_n(\xi)\mathtt{z}_{mn}\\
&=\sum_{m,n= 1}^N(-i)^n\psi_m(s)\varphi_n(\xi)\mathtt{z}_{mn},\quad s\geq 0,\xi\in\mathbb{R}, 
\end{align*}
where we utilized the well known property
\begin{align*}
\widehat{\varphi}_n(\xi)=(-i)^n\varphi_n(\xi),\quad \xi\in\mathbb{R},n\geq 1.
\end{align*} 
Recalling that \eqref{1bis_N} corresponds to \eqref{tricomi det} with data \eqref{data} we conclude by virtue of \eqref{eq:solution} that
\begin{align}\label{U_N}
U_N(t,x)=\sum_{m,n= 1}^N\frac{(-i)^n}{2\pi}\int_{\mathbb{R}} e^{i\xi x}
\int_{0}^{t} C(t,s,|\xi|)\psi_m(s)\varphi_n(\xi)dsd\xi\quad\!\!\!\!\mathtt{z}_{mn},\quad t\geq 0,x\in\mathbb{R}.
\end{align}
To prove that \eqref{U_N} admits a limit in $\mathbb{L}^2(\Omega,\mathcal{F},\mathbb{P})$ as $N$ tends to infinity, we have to verify the convergence of the series
\begin{align*}
\sum_{m,n\geq 1}\left|\frac{(-i)^n}{2\pi}\int_{\mathbb{R}} e^{i\xi x}
\int_{0}^{t} C(t,s,|\xi|)\psi_m(s)\varphi_n(\xi)dsd\xi\right|^2.
\end{align*}
If this is the case, then
\begin{align*}
U(t,x):=\sum_{m,n\geq 1}\frac{(-i)^n}{2\pi}\int_{\mathbb{R}} e^{i\xi x}
\int_{0}^{t} C(t,s,|\xi|)\psi_m(s)\varphi_n(\xi)dsd\xi\quad\!\!\!\!\mathtt{z}_{mn},\quad t\geq 0,x\in\mathbb{R}
\end{align*}
will be solution to \eqref{1bis}. Now, identity
\begin{align*}
\sum_{n\geq 1}\left|\int_{\mathbb{R}}f(\xi)\varphi_n(\xi)d\xi\right|^2=\int_{\mathbb{R}}|f(\xi)|^2d\xi,\quad f\in L^2(\mathbb{R})
\end{align*}
yields
\begin{align*}
&\sum_{m,n\geq 1}\left|\frac{(-i)^n}{2\pi}\int_{\mathbb{R}}\left[e^{i\xi x}\left(\int_0^tC(t,s,|\xi|)\psi_m(s)ds\right)\right]\varphi_n(\xi)d\xi\right|^2\\
&\quad=\sum_{m\geq 1}\sum_{n\geq 1}\left|\frac{(-i)^n}{2\pi}\int_{\mathbb{R}}\left[e^{i\xi x}\left(\int_0^tC(t,s,|\xi|)\psi_m(s)ds\right)\right]\varphi_n(\xi)d\xi\right|^2\\
&\quad=\sum_{m\geq 1}\frac{1}{4\pi^2}\int_{\mathbb{R}}\left|e^{i\xi x}\left(\int_0^tC(t,s,|\xi|)\psi_m(s)ds\right)\right|^2d\xi\\
&\quad=\sum_{m\geq 1}\frac{1}{4\pi^2}\int_{\mathbb{R}}\left|\int_0^tC(t,s,|\xi|)\psi_m(s)ds\right|^2d\xi\\
&\quad=\sum_{m\geq 1}\frac{1}{4\pi^2}\int_{\mathbb{R}}\left|\int_0^{+\infty}C(t,s,|\xi|)1_{[0,t]}(s)\psi_m(s)ds\right|^2d\xi\\
&\quad=\frac{1}{4\pi^2}\int_{\mathbb{R}}\sum_{m\geq 1}\left|\int_0^{+\infty}C(t,s,|\xi|)1_{[0,t]}(s)\psi_m(s)ds\right|^2d\xi.
\end{align*}
Moreover, thanks to 
\begin{align*}
\sum_{m\geq 1}\left|\int_0^{+\infty}g(t)\psi_m(t)dt\right|^2=\int_0^{+\infty}|g(t)|^2d\xi,\quad g\in L^2([0,+\infty[),
\end{align*}
we can write
\begin{align*}
&\sum_{m,n\geq 1}\left|\frac{(-i)^n}{2\pi}\int_{\mathbb{R}}\left[e^{i\xi x}\left(\int_0^tC(t,s,|\xi|)\psi_m(s)ds\right)\right]\varphi_n(\xi)d\xi\right|^2\\
&=\frac{1}{4\pi^2}\int_{\mathbb{R}}\sum_{m\geq 1}\left|\int_0^{+\infty}C(t,s,|\xi|)1_{[0,t]}(s)\psi_m(s)ds\right|^2d\xi\\
&=\frac{1}{4\pi^2}\int_{\mathbb{R}}\int_0^{+\infty}\left|C(t,s,|\xi|)1_{[0,t]}(s)\right|^2dsd\xi\\
&=\frac{1}{4\pi^2}\int_{\mathbb{R}}\int_0^{t}\left|C(t,s,|\xi|)\right|^2ds d\xi.
\end{align*}
This proves the identity in \eqref{norma}. To prove the finiteness of the last expression
we substitute the identity
\begin{align*}
C(t,s,\rho)=-\frac{\pi}{\rho^{2/3}}\left[\mathrm{Bi}(-\rho^{2/3}t)\mathrm{Ai}(-\rho^{2/3}s)-\mathrm{Ai}(-\rho^{2/3}t)\mathrm{Bi}(-\rho^{2/3}s)\right]
\end{align*}
in \eqref{norma} to get
\begin{align}\label{z}
&\frac{1}{2\pi^2}\int_0^{+\infty}\int_0^{t}\left|C(t,s,\rho)\right|^2ds d\rho\nonumber\\
&\quad=\frac{1}{2\pi^2}\int_0^{+\infty}\int_0^{t}\left|\frac{\pi}{\rho^{2/3}}\left[\mathrm{Bi}(-\rho^{2/3}t)\mathrm{Ai}(-\rho^{2/3}s)-\mathrm{Ai}(-\rho^{2/3}t)\mathrm{Bi}(-\rho^{2/3}s)\right]\right|^2ds d\rho\nonumber\\
&\quad=\frac{1}{2}\int_0^{t}\int_0^{+\infty}\frac{1}{\rho^{4/3}}\left|\mathrm{Bi}(-\rho^{2/3}t)\mathrm{Ai}(-\rho^{2/3}s)-\mathrm{Ai}(-\rho^{2/3}t)\mathrm{Bi}(-\rho^{2/3}s)\right|^2 d\rho ds.
\end{align}
The function
\begin{align*}
	(s,\rho)\mapsto \mathrm{Bi}(-\rho^{2/3}t)\mathrm{Ai}(-\rho^{2/3}s)-\mathrm{Ai}(-\rho^{2/3}t)\mathrm{Bi}(-\rho^{2/3}s)
\end{align*}
is continuous and bounded on $[0,t]\times [0,+\infty[$ for all $t\geq 0$ (see Figure \ref{Airyfigure} for the plots of $\mathrm{Ai}$ and $\mathrm{Bi}$ for negative values of their arguments). Therefore, the inner integral is convergent for large values of $\rho$ due to the presence of the factor $\frac{1}{\rho^{4/3}}$. On the other hand, when $\rho$ is close to zero, we have 
\begin{align*}
	\mathrm{Bi}(-\rho^{2/3}t)\mathrm{Ai}(-\rho^{2/3}s)-\mathrm{Ai}(-\rho^{2/3}t)\mathrm{Bi}(-\rho^{2/3}s)\thicksim \frac{2(s-t)}{\sqrt{3}\Gamma(1/3)\Gamma(2/3)}\rho^{2/3}+O(\rho^{4/3})
\end{align*}
and hence
\begin{align*}
\frac{1}{\rho^{4/3}}\left|\mathrm{Bi}(-\rho^{2/3}t)\mathrm{Ai}(-\rho^{2/3}s)-\mathrm{Ai}(-\rho^{2/3}t)\mathrm{Bi}(-\rho^{2/3}s)\right|^2\thicksim \left(\frac{2(s-t)}{\sqrt{3}\Gamma(1/3)\Gamma(2/3)}\right)^2+O(\rho^{2/3})
\end{align*}
thus entailing the finiteness on the inner integral in \eqref{z}.
The almost sure convergence of the series \eqref{solution series}
follows from its
$\mathbb{L}^2(\Omega,\mathcal{F},\mathbb{P})$-convergence and
classical results on random series, e.g. Kolmogorov’s convergence theorem for independent Gaussian series.

\begin{figure}
	\centering
	\begin{subfigure}[c]{0.40\textwidth}
		\IfFileExists{Ai.jpeg}{\includegraphics[width=\linewidth]{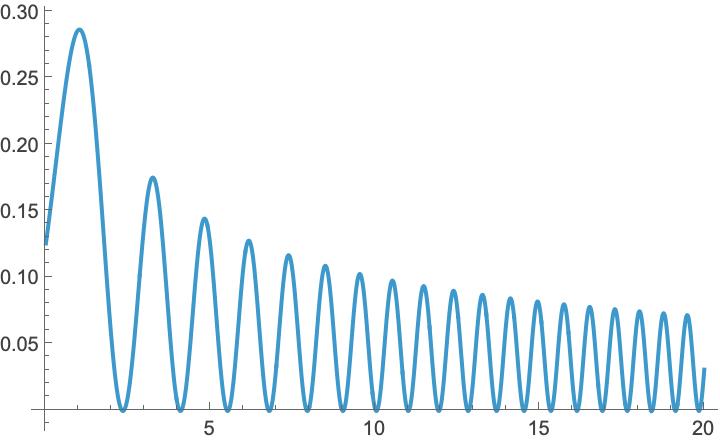}}{\fbox{\parbox[c][3cm][c]{\linewidth}{\centering Plot of $\mathrm{Ai}(-x)$}}}
	\end{subfigure}
	\begin{subfigure}[c]{0.40\textwidth}
		\IfFileExists{Bi.jpeg}{\includegraphics[width=\linewidth]{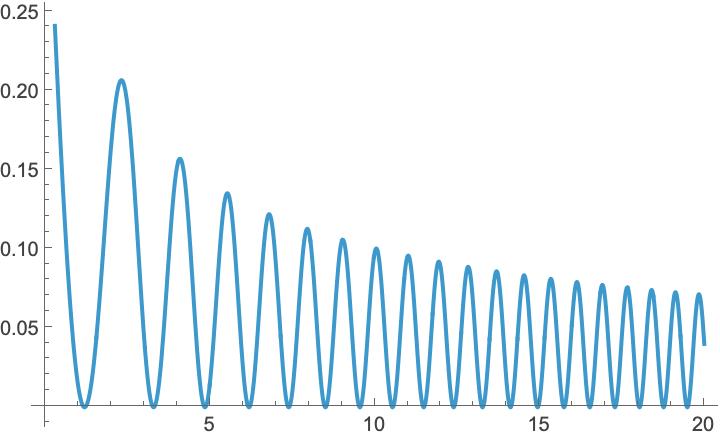}}{\fbox{\parbox[c][3cm][c]{\linewidth}{\centering Plot of $\mathrm{Bi}(-x)$}}}
	\end{subfigure}
	\caption{Plots of $x\mapsto\mathrm{Ai}(-x)$, left figure, and $x\mapsto\mathrm{Bi}(-x)$, right figure, for $x\in [0.+\infty[$.}\label{Airyfigure}
\end{figure}

\section{Proof of Corollary \ref{corollary}}\label{proof corollary}

Fix $t>0$; then,
\begin{align*}
\mathbb{E}[U(t,x)U(t,y)]=&\mathbb{E}[U(t,x)\overline{U(t,y)}]\\
=&\mathbb{E}\left[\sum_{m,n\geq 1}u_{mn}(t,x)\mathtt{z}_{mn}\overline{\sum_{m,n\geq 1}u_{mn}(t,y)\mathtt{z}_{mn}}\right]\\
=&\sum_{m,n\geq 1}u_{mn}(t,x)\overline{u_{mn}(t,y)}\\
=&\sum_{m,n\geq 1}\frac{1}{4\pi^2}\int_{\mathbb{R}}\left[e^{i\xi x}(-i)^n\left(\int_0^tC(t,s,|\xi|)\psi_m(s)ds\right)\right]\varphi_n(\xi)d\xi\\
&\quad\quad\quad\times\int_{\mathbb{R}}\left[e^{-i\xi y}i^n\left(\int_0^tC(t,s,|\xi|)\psi_m(s)ds\right)\right]\varphi_n(\xi)d\xi\\
=&\sum_{m,n\geq 1}\frac{1}{4\pi^2}\int_{\mathbb{R}}\left[e^{i\xi x}\left(\int_0^tC(t,s,|\xi|)\psi_m(s)ds\right)\right]\varphi_n(\xi)d\xi\\
&\quad\quad\quad\times\int_{\mathbb{R}}\left[e^{-i\xi y}\left(\int_0^tC(t,s,|\xi|)\psi_m(s)ds\right)\right]\varphi_n(\xi)d\xi\\
=&\sum_{m\geq 1}\sum_{n\geq 1}\frac{1}{4\pi^2}\int_{\mathbb{R}}\left[e^{i\xi x}\left(\int_0^tC(t,s,|\xi|)\psi_m(s)ds\right)\right]\varphi_n(\xi)d\xi\\
&\quad\quad\quad\times\int_{\mathbb{R}}\left[e^{-i\xi y}\left(\int_0^tC(t,s,|\xi|)\psi_m(s)ds\right)\right]\varphi_n(\xi)d\xi\\
=&\sum_{m\geq 1}\frac{1}{4\pi^2}\int_{\mathbb{R}}e^{i\xi (x-y)}\left(\int_0^tC(t,s,|\xi|)\psi_m(s)ds\right)^2d\xi\\
=&\frac{1}{4\pi^2}\int_{\mathbb{R}}e^{i\xi (x-y)}\int_0^t|C(t,s,|\xi|)|^2dsd\xi\\
=&\frac{1}{4\pi^2}\int_{\mathbb{R}}\cos(\xi (x-y))\int_0^t|C(t,s,|\xi|)|^2dsd\xi\\
=&\frac{1}{2\pi^2}\int_0^{+\infty}\cos(\rho |x-y|)\int_0^t|C(t,s,\rho)|^2dsd\rho.
\end{align*}
This proves formula \eqref{cov}. 

Now, according to \cite{Orsingher} (see also \cite{Pickands}) to prove the continuity of the stationary process $\{U(t,x)\}_{x\in\mathbb{R}}$ it is sufficient to show that its correlation function
\begin{align*}
r_t(h)=\frac{\mathbb{E}[U(t,x+h)U(t,x)]}{\mathbb{V}[U(t,x)]},\quad h\in\mathbb{R}
\end{align*} 
verifies 
\begin{align}\label{r}
r_t(h)=1-\alpha_t|h|^{\beta_t}+o(|h|)\quad\mbox{ as $h$ tends to zero},
\end{align}
for some $\beta_t\in(0,2]$ and $\alpha_t\in\mathbb{R}$. Notice that by virtue of formula \eqref{cov} we can write
\begin{align*}
r_t(h)&=\frac{\frac{1}{2\pi^2}\int_0^{+\infty}\cos(\rho |h|)\int_0^t|C(t,s,\rho)|^2dsd\rho}{\frac{1}{2\pi^2}\int_0^{+\infty}\int_0^t|C(t,s,\rho)|^2dsd\rho}\\
&=1-\frac{\int_0^{+\infty}(1-\cos(\rho |h|))\int_0^t|C(t,s,\rho)|^2dsd\rho}{\int_0^{+\infty}\int_0^t|C(t,s,\rho)|^2dsd\rho}\\
&=1-\hat{\alpha_t}\cdot\int_0^{+\infty}(1-\cos(\rho |h|))\int_0^t|C(t,s,\rho)|^2dsd\rho,
\end{align*}
with
\begin{align*}
\hat{\alpha_t}:=\frac{1}{\int_0^{+\infty}\int_0^t|C(t,s,\rho)|^2dsd\rho}.
\end{align*}
We thus have to investigate the expression
\begin{align*}
\int_0^{+\infty}(1-\cos(\rho |h|))\int_0^t|C(t,s,\rho)|^2dsd\rho.
\end{align*}
To establish our claim we need the following crucial auxiliary fact.

\begin{lemma}\label{as}
Fix \(t>0\), and set
\[
  K_t(\rho):=\int_0^t |C(t,s,\rho)|^2\,ds ,
\]
where
\[
C(t,s,\rho)
=
-\pi\rho^{-2/3}
\Big[
 \Bi(-\rho^{2/3}t)\Ai(-\rho^{2/3}s)
 -
 \Ai(-\rho^{2/3}t)\Bi(-\rho^{2/3}s)
\Big].
\]
Then
\[
  K_t(\rho)=\rho^{-2}+o(\rho^{-2})
  \qquad \text{as } \rho\to+\infty.
\]
Equivalently,
\[
  \rho^2K_t(\rho)\longrightarrow 1 .
\]
\end{lemma}

\begin{proof}
Fix \(t>0\). We shall prove that
\[
  \rho^2\int_0^t |C(t,s,\rho)|^2\,ds \longrightarrow 1 .
\]
We use the standard Airy asymptotics on the negative real axis. Namely, if
\[
  \theta(x):=\frac23 x^{3/2}+\frac{\pi}{4},
  \qquad x>0,
\]
then, as \(x\to+\infty\),
\[
  \Ai(-x)
  =
  \pi^{-1/2}x^{-1/4}
  \bigl(\sin\theta(x)+r_A(x)\bigr),
\]
and
\[
  \Bi(-x)
  =
  \pi^{-1/2}x^{-1/4}
  \bigl(\cos\theta(x)+r_B(x)\bigr),
\]
where
\[
  r_A(x)\to0,\qquad r_B(x)\to0
  \qquad \text{as } x\to+\infty.
\]
Equivalently, if
\[
  \varepsilon_L
  :=
  \sup_{x\ge L}\bigl(|r_A(x)|+|r_B(x)|\bigr),
\]
then
\[
  \varepsilon_L\to0
  \qquad \text{as } L\to+\infty.
\]
Let \(L>1\) be fixed. For \(\rho\) large enough we have
\[
  L\rho^{-2/3}<t.
\]
We split
\[
  \rho^2K_t(\rho)
  =
  \rho^2\int_0^{L\rho^{-2/3}} |C(t,s,\rho)|^2\,ds
  +
  \rho^2\int_{L\rho^{-2/3}}^t |C(t,s,\rho)|^2\,ds
  =: I_{\mathrm{small}}(\rho,L)+I_{\mathrm{bulk}}(\rho,L).
\]
We first show that the small interval gives no contribution. On
\[
  0\le s\le L\rho^{-2/3},
\]
the quantities \(\rho^{2/3}s\) remain in the compact interval \([0,L]\). Hence
\(\Ai(-\rho^{2/3}s)\) and \(\Bi(-\rho^{2/3}s)\) are bounded by a constant depending only on \(L\). On the other hand,
\[
  \Ai(-\rho^{2/3}t)=O((\rho^{2/3}t)^{-1/4})
  =O(\rho^{-1/6}),
\]
and similarly
\[
  \Bi(-\rho^{2/3}t)=O(\rho^{-1/6}).
\]
Therefore
\[
\begin{aligned}
 &\Big|
 \Bi(-\rho^{2/3}t)\Ai(-\rho^{2/3}s)
 -
 \Ai(-\rho^{2/3}t)\Bi(-\rho^{2/3}s)
 \Big|
 \le C_L\rho^{-1/6}.
\end{aligned}
\]
Consequently,
\[
  |C(t,s,\rho)|^2
  \le C_L \rho^{-4/3}\rho^{-1/3}
  =
  C_L\rho^{-5/3}.
\]
Thus
\[
  I_{\mathrm{small}}(\rho,L)
  \le
  \rho^2
  \int_0^{L\rho^{-2/3}} C_L\rho^{-5/3}\,ds
  =
  C_L L \rho^{-1/3}
  \longrightarrow 0
\]
as \(\rho\to+\infty\), for every fixed \(L\).\\
We now study \(I_{\mathrm{bulk}}(\rho,L)\). Put
\[
  x_t:=\rho^{2/3}t,\qquad x_s:=\rho^{2/3}s .
\]
On the interval \(s\in[L\rho^{-2/3},t]\), both \(x_s\) and \(x_t\) are at least \(L\), for \(\rho\) sufficiently large. Therefore the Airy asymptotics above apply uniformly. Define
\[
\begin{aligned}
D(t,s,\rho)
&:=
\Bi(-x_t)\Ai(-x_s)
-
\Ai(-x_t)\Bi(-x_s).
\end{aligned}
\]
Using the asymptotic formulae, we get
\[
\begin{aligned}
D(t,s,\rho)
&=
\pi^{-1}(x_tx_s)^{-1/4}
\Big[
 \cos\theta(x_t)\sin\theta(x_s)
 -
 \sin\theta(x_t)\cos\theta(x_s)
 +
 R_L(t,s,\rho)
\Big],
\end{aligned}
\]
where, uniformly for \(s\in[L\rho^{-2/3},t]\),
\[
  |R_L(t,s,\rho)|\le C\varepsilon_L .
\]
Since
\[
  \cos\theta(x_t)\sin\theta(x_s)
  -
  \sin\theta(x_t)\cos\theta(x_s)
  =
  \sin(\theta(x_s)-\theta(x_t)),
\]
and since
\[
  \theta(x_t)-\theta(x_s)
  =
  \frac23 \rho(t^{3/2}-s^{3/2}),
\]
we obtain
\[
D(t,s,\rho)
=
\pi^{-1}\rho^{-1/3}(ts)^{-1/4}
\Big[
 -\sin\Big(\frac23\rho(t^{3/2}-s^{3/2})\Big)
 +
 R_L(t,s,\rho)
\Big].
\]
Multiplying by the prefactor in \(C\), namely \(C=-\pi\rho^{-2/3}D\), gives
\[
C(t,s,\rho)
=
\rho^{-1}(ts)^{-1/4}
\Big[
 \sin\Big(\frac23\rho(t^{3/2}-s^{3/2})\Big)
 +
 \widetilde R_L(t,s,\rho)
\Big],
\]
with
\[
  |\widetilde R_L(t,s,\rho)|\le C\varepsilon_L
\]
uniformly on the same interval. Therefore
\[
\begin{aligned}
I_{\mathrm{bulk}}(\rho,L)
&=
\int_{L\rho^{-2/3}}^t
(ts)^{-1/2}
\Big|
 \sin\Big(\frac23\rho(t^{3/2}-s^{3/2})\Big)
 +
 \widetilde R_L(t,s,\rho)
\Big|^2\,ds .
\end{aligned}
\]
It follows that
\[
\begin{aligned}
I_{\mathrm{bulk}}(\rho,L)
&=
J(\rho,L)+E(\rho,L),
\end{aligned}
\]
where
\[
  J(\rho,L)
  :=
  \int_{L\rho^{-2/3}}^t
  (ts)^{-1/2}
  \sin^2\Big(\frac23\rho(t^{3/2}-s^{3/2})\Big)\,ds,
\]
and
\[
  |E(\rho,L)|
  \le
  C\varepsilon_L
  \int_{L\rho^{-2/3}}^t (ts)^{-1/2}\,ds .
\]
But
\[
  \int_{L\rho^{-2/3}}^t (ts)^{-1/2}\,ds
  =
  t^{-1/2}
  \int_{L\rho^{-2/3}}^t s^{-1/2}\,ds
  \le 2.
\]
Hence
\[
  |E(\rho,L)|\le C\varepsilon_L .
\]
It remains to compute the limit of \(J(\rho,L)\). Using
\[
  \sin^2 y=\frac12(1-\cos(2y)),
\]
we write
\[
\begin{aligned}
J(\rho,L)
&=
\frac12 t^{-1/2}
\int_{L\rho^{-2/3}}^t s^{-1/2}\,ds        \\
&\quad
-
\frac12 t^{-1/2}
\int_{L\rho^{-2/3}}^t
s^{-1/2}
\cos\Big(\frac43\rho(t^{3/2}-s^{3/2})\Big)\,ds .
\end{aligned}
\]
The non-oscillatory part is
\[
  \frac12 t^{-1/2}
  \int_{L\rho^{-2/3}}^t s^{-1/2}\,ds
  =
  t^{-1/2}
  \left(\sqrt t-\sqrt L\,\rho^{-1/3}\right)
  \longrightarrow 1 .
\]
We claim that the oscillatory part tends to zero. Set
\[
  u=\frac23(t^{3/2}-s^{3/2}).
\]
Then
\[
  du=-s^{1/2}\,ds,
  \qquad
  s=\left(t^{3/2}-\frac32u\right)^{2/3},
\]
and therefore
\[
  s^{-1/2}\,ds
  =
  -s^{-1}\,du
  =
  -\left(t^{3/2}-\frac32u\right)^{-2/3}\,du .
\]
Thus
\[
\begin{aligned}
&\int_{L\rho^{-2/3}}^t
s^{-1/2}
\cos\Big(\frac43\rho(t^{3/2}-s^{3/2})\Big)\,ds      \\
&\qquad =
\int_0^{\frac23(t^{3/2}-L^{3/2}\rho^{-1})}
\left(t^{3/2}-\frac32u\right)^{-2/3}
\cos(2\rho u)\,du .
\end{aligned}
\]
The function
\[
  g(u):=
  \left(t^{3/2}-\frac32u\right)^{-2/3}
\]
belongs to
\[
  L^1\left(0,\frac23t^{3/2}\right),
\]
because its singularity at the endpoint is of order \(2/3<1\). Hence, by the Riemann--Lebesgue lemma,
\[
  \int_0^{\frac23t^{3/2}}
  g(u)\cos(2\rho u)\,du
  \longrightarrow 0 .
\]
Moreover,
\[
\begin{aligned}
&\left|
\int_{\frac23(t^{3/2}-L^{3/2}\rho^{-1})}^{\frac23t^{3/2}}
g(u)\cos(2\rho u)\,du
\right|        \\
&\qquad \le
\int_{\frac23(t^{3/2}-L^{3/2}\rho^{-1})}^{\frac23t^{3/2}}
g(u)\,du
\longrightarrow 0 .
\end{aligned}
\]
Therefore the oscillatory part tends to zero, and so
\[
  J(\rho,L)\longrightarrow 1
  \qquad \text{as } \rho\to+\infty .
\]
We have shown that, for every fixed \(L>1\),
\[
  \limsup_{\rho\to+\infty}
  \bigl|\rho^2K_t(\rho)-1\bigr|
  \le C\varepsilon_L .
\]
Finally, since \(\varepsilon_L\to0\) as \(L\to+\infty\), we conclude that
\[
  \rho^2K_t(\rho)\longrightarrow 1 .
\]
Equivalently,
\[
  K_t(\rho)=\rho^{-2}+o(\rho^{-2})
  \qquad \text{as } \rho\to+\infty .
\]
The proof is complete.
\end{proof}

\begin{remark}[Control of the Airy remainder terms]\label{cross}
The preceding proof avoids the decomposition into a principal part and a separate remainder term on a fixed interval
\([s_\rho,t]\).  The reason for introducing the auxiliary parameter \(L\) is precisely to make the Airy remainders uniformly small on the whole bulk region
\[
   L\rho^{-2/3}\le s\le t.
\]
Indeed, on this region we obtained
\[
C(t,s,\rho)
=
\rho^{-1}(ts)^{-1/4}
\left[
S(t,s,\rho)+\widetilde R_L(t,s,\rho)
\right],
\]
where
\[
   S(t,s,\rho)
   :=
   \sin\!\left(\frac23\rho(t^{3/2}-s^{3/2})\right),
   \qquad
   |\widetilde R_L(t,s,\rho)|\le C\varepsilon_L,
\]
and
\[
   \varepsilon_L:=\sup_{x\ge L}\big(|r_A(x)|+|r_B(x)|\big)\to0
   \qquad\text{as }L\to+\infty.
\]
Therefore the whole contribution of the cross term and of the square of the remainder is estimated at once by
\[
\begin{aligned}
&\left|
\int_{L\rho^{-2/3}}^t
(ts)^{-1/2}
\left(
 |S(t,s,\rho)+\widetilde R_L(t,s,\rho)|^2
 -|S(t,s,\rho)|^2
\right)ds
\right|                                      \\
&\qquad\le
C\varepsilon_L
\int_{L\rho^{-2/3}}^t (ts)^{-1/2}\,ds
\le C\varepsilon_L .
\end{aligned}
\]
Thus the error terms are not treated by proving a separate oscillatory estimate for each cross term.  They are instead bounded uniformly, for large \(\rho\), by a quantity that can be made arbitrarily small by choosing \(L\) large.  This is the point of the two-step argument: first let \(\rho\to+\infty\) with \(L\) fixed, and then let \(L\to+\infty\).
\end{remark}

We are now ready to prove \eqref{r} and hence to complete the proof of Corollary \ref{corollary}.
By virtue of the classical identity
	\begin{align*}
		\int_0^{+\infty} \frac{1-\cos(\rho h)}{\rho^2}d\rho=\frac{\pi}{2}|h|,\quad h\in\mathbb{R}
	\end{align*}
we can write
\begin{align*}
		\int_0^{+\infty}(1-\cos(\rho h))K_t(\rho)d\rho&=\int_0^{+\infty}(1-\cos(\rho h))(\rho^{-2}+\tilde{K}_t(\rho))d\rho\\
		&=\frac{\pi}{2}|h|+\int_0^{+\infty}(1-\cos(\rho h))\tilde{K}_t(\rho)d\rho
\end{align*}	
where $\tilde{K}_t(\rho):=K_t(\rho)-\rho^{-2}$. By Lemma \ref{as},
$\tilde{K}_t(\rho)=o(\rho^{-2})$ as $\rho\to+\infty$. In particular, for any $\varepsilon>0$ there exists $\rho_0$ such that $|\tilde{K}_t(\rho)|\leq \varepsilon\rho^{-2}$ for all $\rho\geq \rho_0$. With such $\rho_0$ we split the last integral as
\begin{align*}
	\int_0^{+\infty}(1-\cos(\rho h))\tilde{K}_t(\rho)d\rho&=\int_0^{\rho_0}(1-\cos(\rho h))\tilde{K}_t(\rho)d\rho+\int_{\rho_0}^{+\infty}(1-\cos(\rho h))\tilde{K}_t(\rho)d\rho\\
	&\leq \frac{|h|^2}{2}\int_0^{\rho_0}\rho^2|\tilde{K}_t(\rho)|d\rho+\varepsilon\frac{\pi}{2}|h|.
\end{align*}
This shows that
\begin{align*}
	r_t(h)=1-\alpha_t|h|+o(|h|)\mbox{ as }h\to 0
\end{align*}
with
\begin{align*}
	\alpha_t:=\frac{\pi}{2\int_0^{+\infty}K_t(\rho)d\rho}. 
\end{align*}

\section{Proof of Corollary \ref{corollary2}}\label{proof corollary2}

We prove the estimate on an arbitrary compact time interval \([0,T]\).  By
Gaussianity, this second-moment estimate implies the asserted Hölder
regularity through Kolmogorov's continuity theorem.

We shall use the following elementary consequences of the Airy asymptotics on
the negative real axis.  There exists a constant \(C>0\) such that, for every
\(x\ge0\),
\begin{align}
  |\Ai(-x)|+|\Bi(-x)| &\le C(1+x)^{-1/4}, \label{eq:airy-global-0}\\
  |\Ai'(-x)|+|\Bi'(-x)| &\le C(1+x)^{1/4}. \label{eq:airy-global-1}
\end{align}
The estimates are standard: for bounded \(x\) they follow from smoothness of
\(\Ai\) and \(\Bi\), while for large \(x\) they are exactly the usual
oscillatory Airy estimates.

Let
\[
  0\le s\le t\le T,
  \qquad
  \delta:=t-s.
\]
The cases \(t=s\) and \(t=0\) are trivial.  Moreover, it is enough to prove the
estimate for \(0<\delta\le1\); the case \(\delta>1\) is absorbed by increasing
\(C_T\), since \(\sup_{0\le u\le T}\mathbb E|U(u,x)|^2<+\infty\), by the variance
formula \eqref{norma} and the same kernel bounds used below.  Thus we assume
\(0<\delta\le1\).\\
From the covariance formula, or equivalently from the isometry used in the
proof of Theorem \ref{main theorem}, we have
\begin{align}
  \mathbb E|U(t,x)-U(s,x)|^2
  &=\frac{1}{2\pi^2}
    \int_0^\infty\int_0^s
    |C(t,r,\rho)-C(s,r,\rho)|^2\,dr\,d\rho          \notag\\
  &\quad
    +\frac{1}{2\pi^2}
    \int_0^\infty\int_s^t
    |C(t,r,\rho)|^2\,dr\,d\rho                         \notag\\
  &=: \frac{1}{2\pi^2}(I_1+I_2).
  \label{eq:variance-split}
\end{align}
We prove that \(I_1+I_2\le C_T\delta\). Set
\[
  \rho_t:=t^{-3/2}.
\]
For \(0<\rho\le \rho_t\), all Airy arguments \(\rho^{2/3}\tau\), with
\(0\le \tau\le t\), remain in a fixed bounded interval.  Differentiating the
formula for \(C\) in the first time variable gives
\begin{align*}
\partial_\tau C(\tau,r,\rho)
 =\pi\Big[\Bi'(-\rho^{2/3}\tau)\Ai(-\rho^{2/3}r)
       -\Ai'(-\rho^{2/3}\tau)\Bi(-\rho^{2/3}r)\Big],
\end{align*}
which is uniformly bounded in this small-argument region.  Since
\(C(r,r,\rho)=0\), we obtain
\begin{align}
  |C(t,r,\rho)|\le C_T|t-r|,
  \qquad
  |C(t,r,\rho)-C(s,r,\rho)|\le C_T\delta.
  \label{eq:small-rho-bounds}
\end{align}
Consequently,
\begin{align}
\int_0^{\rho_t}\int_0^s
|C(t,r,\rho)-C(s,r,\rho)|^2\,dr\,d\rho
&\le C_T\rho_t s\delta^2
 \le C_T t^{-3/2}t\delta^2                    \notag\\
&\le C_T\delta,                                  \label{eq:small-rho-I1}
\end{align}
and
\begin{align}
\int_0^{\rho_t}\int_s^t |C(t,r,\rho)|^2\,dr\,d\rho
&\le C_T\rho_t\int_s^t (t-r)^2\,dr              \notag\\
&\le C_Tt^{-3/2}\delta^3
 \le C_T\delta.                                  \label{eq:small-rho-I2}
\end{align}
Here we used \(s\le t\), \(\delta\le t\), and \(t\le T\).\\
It remains to estimate the contribution of \(\rho\ge \rho_t\).  We first treat
\(I_1\).  Split the \(r\)-integration into
\[
  0\le r\le \rho^{-2/3},
  \qquad
  \rho^{-2/3}\le r\le s.
\]
The first interval corresponds to the region where the second Airy argument
\(\rho^{2/3}r\) is small.  In this region \(\Ai(-\rho^{2/3}r)\) and
\(\Bi(-\rho^{2/3}r)\) are uniformly bounded.  From the formula for \(C\) we
always have
\[
 |C(t,r,\rho)-C(s,r,\rho)|\le C\rho^{-2/3}.
\]
If \(\rho\ge1\), the derivative estimate \eqref{eq:airy-global-1} also gives
\[
 |C(t,r,\rho)-C(s,r,\rho)|\le C_T\delta\rho^{1/6}.
\]
The part \(\rho_t\le\rho\le1\) contributes at most
\[
  C_T\delta^2\int_{\rho_t}^{1}\rho^{-2/3}\,d\rho\le C_T\delta.
\]
For the remaining part \(\rho\ge\max\{\rho_t,1\}\), we get
\begin{align}
&\int_{\max\{\rho_t,1\}}^\infty\int_0^{\min\{s,\rho^{-2/3}\}}
|C(t,r,\rho)-C(s,r,\rho)|^2\,dr\,d\rho             \notag\\
&\qquad\le
C_T\int_0^\infty
\rho^{-2/3}\min\{\rho^{-4/3},\delta^2\rho^{1/3}\}\,d\rho
\le C_T\delta.                                      \label{eq:I1-small-r}
\end{align}
Indeed, splitting the last integral at \(\rho=\delta^{-6/5}\) gives a bound
\(C\delta^{6/5}\le C\delta\), because \(0<\delta\le1\).\\
We now consider the second region, \(\rho^{-2/3}\le r\le s\).  Then
\(\rho^{2/3}r\ge1\), and hence also \(\rho^{2/3}s\ge1\).  Using
\eqref{eq:airy-global-0} and \eqref{eq:airy-global-1}, we obtain
\begin{align}
 |C(t,r,\rho)-C(s,r,\rho)|
 \le C_T r^{-1/4}
 \min\{\rho^{-1}s^{-1/4},\delta t^{1/4}\}.
 \label{eq:large-r-difference}
\end{align}
The first term in the minimum follows by estimating the difference by the sum
of the two values at \(t\) and \(s\); the second follows by the mean-value
formula in the first time variable.  Squaring \eqref{eq:large-r-difference}
and integrating first with respect to \(\rho\), we use the elementary identity
\[
  \int_0^\infty \min\{A\rho^{-2},B\}\,d\rho
  \le 2\sqrt{AB},
  \qquad A,B>0.
\]
With \(A=s^{-1/2}\) and \(B=\delta^2t^{1/2}\), this gives
\begin{align}
&\int_{\rho_t}^\infty\int_{\min\{s,\rho^{-2/3}\}}^s
|C(t,r,\rho)-C(s,r,\rho)|^2\,dr\,d\rho              \notag\\
&\qquad\le
C_T\int_0^s r^{-1/2}\,dr
\int_0^\infty \min\{\rho^{-2}s^{-1/2},\delta^2t^{1/2}\}\,d\rho
                                                                  \notag\\
&\qquad\le
C_T s^{1/2}\delta t^{1/4}s^{-1/4}
\le C_T\delta.                                      \label{eq:I1-large-r}
\end{align}
Combining \eqref{eq:small-rho-I1}, \eqref{eq:I1-small-r}, and
\eqref{eq:I1-large-r}, we obtain
\begin{align}
  I_1\le C_T\delta.                                  \label{eq:I1-final}
\end{align}
It remains to estimate \(I_2\) in the large-frequency region.  For
\(\rho\ge\rho_t\) and \(0<r\le t\), estimate \eqref{eq:airy-global-0} gives
\begin{align*}
 |C(t,r,\rho)|\le C_T\rho^{-1}(tr)^{-1/4}.
\end{align*}
Therefore
\begin{align*}
\int_{\rho_t}^\infty\int_s^t |C(t,r,\rho)|^2\,dr\,d\rho
&\le C_T\int_{\rho_t}^\infty \rho^{-2}\,d\rho
       \int_s^t (tr)^{-1/2}\,dr                    \\
&\le C_T\rho_t^{-1}t^{-1/2}(\sqrt t-\sqrt s)        \\
&\le C_Tt^{3/2}t^{-1/2}\frac{t-s}{\sqrt t+\sqrt s}  \\
&\le C_Tt^{1/2}\delta
\le C_T\delta.                                      
\end{align*}
Together with \eqref{eq:small-rho-I2}, this proves
\begin{align}
  I_2\le C_T\delta.                                  \label{eq:I2-final}
\end{align}
From \eqref{eq:variance-split}, \eqref{eq:I1-final}, and \eqref{eq:I2-final},
we conclude that
\begin{align*}
  \mathbb E|U(t,x)-U(s,x)|^2\le C_T|t-s|,
  \qquad 0\le s,t\le T.
\end{align*}
Since the increment is Gaussian, for every \(p\ge2\),
\begin{align*}
 \mathbb E|U(t,x)-U(s,x)|^p
 \le C_p |t-s|^{p/2}.
\end{align*}
Choosing \(p>2\) and applying Kolmogorov's continuity theorem gives a
modification whose time paths are Hölder continuous of every order
\(\gamma<1/2\).  This proves Corollary \ref{corollary2}.

\noindent\textbf{Author Contributions}: All authors contributed equally.

\noindent\textbf{Data availability}: No datasets were generated or analysed during the current study.

\noindent\textbf{Conflict of interest}: The authors declare no competing interests.\\

\bibliographystyle{plain}
\bibliography{Additive_Tricomi}

\end{document}